\documentclass[sn-mathphys]{sn-jnl}% Math and Physical Sciences Reference Style
\jyear{2021}%

\theoremstyle{thmstyleone}%
\newtheorem{theorem}{Theorem}%  meant for continuous numbers
\newtheorem{proposition}[theorem]{Proposition}% 
\newtheorem{corollary}[theorem]{Corollary}% 
\newtheorem{lemma}[theorem]{Lemma}% 

\theoremstyle{thmstyletwo}%
\newtheorem{example}{Example}%
\newtheorem{remark}{Remark}%

\theoremstyle{thmstylethree}%
\newtheorem{definition}{Definition}%

\begin{document}

\title[Interior Dynamics of Fatou Sets]{Interior Dynamics of Fatou Sets}


\author{\fnm{Mi} \sur{Hu}}\email{mi.hu@unipr.it}



\affil{\orgdiv{Department of Mathematical, Physical and Computer Sciences}, \orgname{University of Parma}, \orgaddress{\street{Parco Area delle Scienze}, \city{Parma}, \postcode{43124},  \country{ Italy}}}

%\affil[2]{\orgdiv{Department}, \orgname{Organization}, \orgaddress{\street{Street}, \city{City}, \postcode{10587}, \state{State}, \country{Country}}}

%\affil[3]{\orgdiv{Department}, \orgname{Organization}, \orgaddress{\street{Street}, \city{City}, \postcode{610101}, \state{State}, \country{Country}}}


\abstract{In this paper, we investigate the precise behavior of orbits inside attracting basins. Let $f$ be a holomorphic polynomial of degree $m\geq2$ in $\mathbb{C}$, $\mathcal {A}(p)$ be the basin of attraction of an attracting fixed point $p$ of $f$, and $\Omega_i (i=1, 2, \cdots)$ be the connected components of $\mathcal{A}(p)$. We prove that there is a constant $C$ so that for every point $z_0$ inside any $\Omega_i$, there exists a point $q\in \cup_k f^{-k}(p)$ inside $\Omega_i$ such that $d_{\Omega_i}(z_0, q)\leq C$, where $d_{\Omega_i}$ is the Kobayashi distance on $\Omega_i.$}




\maketitle

\section{Introduction}\label{sec1}

	A general goal in discrete dynamical systems is to qualitatively and quantitatively describe the possible dynamical behaviour under iteration of maps satisfying certain conditions. They may be algebraic (e.g., polynomials or rational maps) or analytic (e.g., smooth, symplectic, or holomorphic self-maps).

Let $\hat{\mathbb{C}}=\mathbb{C}\cup\{\infty\}, f: \hat{\mathbb{C}} \rightarrow \hat{\mathbb{C}}$ be a nonconstant holomorphic map, and $f^{ n}: \hat{\mathbb{C}} \rightarrow \hat{\mathbb{C}}$ be its $n$-fold iterate. In complex dynamics, two crucial disjoint invariant sets are associated with $f$, the {\sl Julia set} and the {\sl Fatou set} \cite{RefM}, which partition the sphere $\hat{\mathbb{C}}$.
The Fatou set of $f$ is defined as the largest open set where the family of iterates is locally normal. In other words, for any point $z\in \hat{\mathbb{C}}$, there exists some neighborhood $U$ of $z$ so that the sequence of iterates of the map restricted to $U$ forms a normal family, so the iterates are well behaved. 
The complement of the Fatou set is called the Julia set. 
The connected components of the Fatou set of $f$ are called {\sl Fatou components}. 
A Fatou component $\Omega\subset \hat{\mathbb{C}}$ of $f$ is {\sl invariant} if $f(\Omega)=\Omega$. 
For $z\in \hat{\mathbb{C}}$, the set $\{z_n\}=\{z_1=f(z_0), z_2=f^2(z_0), \cdots\}$ is called the orbit of the point $z=z_0$. 
If $z_N=z_0$ for some integer $N$, we say that $z_0$ is a periodic point of $f$.
If $N=1$, then $z_0$ is a fixed point of $f.$ 

At the beginning of the $20$th century, Fatou \cite{RefM} classified all possible invariant Fatou components of rational functions on the Riemann sphere. He proved that only the following three cases can occur:
\begin{enumerate}
	\item (attracting case) $\Omega$ contains a fixed point $p$ and the orbit of every point in $\Omega$ converges to $p$. 
	\item (parabolic case) $b\Omega$ contains a fixed point $p$ and the orbit of every point in $\Omega$ converges to $p$.  
	\item (rotation domain) $\Omega$ is conformally equivalent to a unit disk or an annulus and the map is conjugate to an irrational rotation.
\end{enumerate}

The classification of Fatou components was completed in the '80s when Sullivan proved that every Fatou component of a rational map is preperiodic, i.e., there are $n,m \in \mathbb{N}$ such that $f^{n+m}(\Omega)=f^m(\Omega)$. For more details and results, we refer the reader to \cite{RefB}, \cite{RefCG}, and \cite{RefM}.

However, there is no detailed study until now of the more precise behavior of orbits inside the Fatou set. 
For example, let $\mathcal{A}(p):=\{z\in \mathbb{C}; f^n(z)\rightarrow p\}$ be the {\sl basin of attraction} of an attracting fixed point $p$. 
One can ask when $z_0$ is close to $\partial\mathcal{A}(p)$, what orbits $\{z_n\}$ of $z_0$ going from $z_0$ to near the attracting fixed point $p$ look like, or how many the iterations are. 
One application is arising from Newton's method in \cite{RefB}. 
It is of practical interest to know how many times Newton's method must be iterated to get the desired approximation of the root.

These kinds of questions are the main topics of this paper. In the second section, we study the dynamics of holomorphic polynomials on the attracting basin and give our main theorems. Our practical tool uses conjugation to consider the orbits of Blaschke products on the unit disk instead of considering the orbits of polynomials on the attracting basin.

\section{Dynamics of holomorphic polynomials inside an attracting basin}\label{sec2}

	Let $f: \hat{\mathbb{C}}\rightarrow\hat{\mathbb{C}}$ be a holomorphic map and
$\Omega$ be an invariant Fatou component. It follows from the
classification above in section \ref{sec1} that we have a complete understanding of the long-term behaviour of all orbits in $\Omega$, but how precisely the
iterates of a point move inside $\Omega$ is presently still unknown.
This section will more precisely describe how iterates of a
point move inside an invariant Fatou component and the whole basin of attraction. Then we give our main results in Theorem \ref{theom1} and Theorem \ref{the3}. 
%\section{This is an example for first level head---section head}\label{sec3}

\subsection{The Kobayashi metric}\label{subsec2}
	\begin{definition}\label{def1}
	Let $\hat{\Omega}\subset\mathbb C$ be a domain. We choose a point $z\in \hat{\Omega}$ and a vector $\xi$ which is tangent to the plane at the point $z.$ Let $\triangle$ denote the unit disk in the complex plane.
	We define the {\em Kobayashi metric}
	$$
	F_{\hat{\Omega}}(z, \xi):=\inf\{\lambda>0 : \exists f: \triangle\stackrel{hol}{\longrightarrow} \hat{\Omega}, f(0)=z, \lambda f'(0)=\xi\}.
	$$

	Let $\gamma: [0, 1]\rightarrow \hat{\Omega}$ be a piecewise smooth curve.
	The {\em Kobayashi length} of $\gamma$ is defined to be 
	$$ L_{\hat{\Omega}} (\gamma)=\int_{\gamma} F_{\hat{\Omega}}(z, \xi) \lvert dz\rvert=\int_{0}^{1}F_{\hat{\Omega}}\big(\gamma(t), \gamma'(t)\big)\lvert \gamma'(t)\rvert dt.$$

	For any two points $z_1$ and $z_2$ in $\hat{\Omega}$, the {\em Kobayashi distance} between $z_1$ and $z_2$ is defined to be 
	$$d_{\hat{\Omega}}(z_1, z_2)=\inf\{L_{\hat{\Omega}} (\gamma): \gamma ~ \text{is a piecewise smooth curve connecting} ~z_1~ \text{and} ~z_2 \}.$$

	Note that $d_{\hat{\Omega}}(z_1, z_2)$ is defined where $z_1, z_2$ are in the same connected component of $\hat{\Omega}.$

	Let $d_E(z_1, z_2)$ denote the Euclidean metric distance for any two points $z_1, z_2\in\triangle.$ 
	
\end{definition}

We know that if $\hat{\Omega}=\triangle$, then the Kobayashi metric is the same as the Poincare metric (on page 19 in \cite{RefM}), i.e.,
$$
F_{\triangle}(z, \xi)=\frac{\lvert \xi\rvert}{1-\lvert z\rvert^2}.
$$

\begin{proposition}[The distance decreasing property of the Kobayashi Metric \cite{RefK}]\label{pro2}
	Suppose $\Omega_1, \Omega_2$ are domains in $\mathbb C$, $z, \omega\in \Omega_1, \xi\in\mathbb C,$ and $f:\Omega_1\rightarrow\Omega_2$ is holomorphic. Then 
	$$F_{\Omega_2}(f(z
	), f'(z)\xi)\leq F_{\Omega_1}(z, \xi), ~~~d_{\Omega_2}(f(z), f(\omega))\leq d_{\Omega_1}(z,\omega).$$
\end{proposition}

	\begin{corollary}\label{cor3}
	Suppose $\Omega_1\subseteq\Omega_2\subseteq\mathbb C.$ Then for any $z, \omega\in \Omega_1$ and $\xi\in\mathbb C,$ we have 
	$$F_{\Omega_2}(z, \xi)\leq F_{\Omega_1}(z, \xi), ~~~~d_{\Omega_2}(z, \omega)\leq d_{\Omega_1}(z,\omega).$$
\end{corollary}

\subsection{Main results about orbits inside the attracting basin}

Let $f:\mathbb{C}\rightarrow\mathbb{C}$ be a polynomial of degree $N$ and $\Omega$ be
an immediate basin of attraction that contains the attracting fixed point $p$ of $f.$ Then $\Omega$ is a connected component of $\mathcal{A}(p)$. 
In addition, an immediate basin of attraction of a holomorphic polynomial is simply connected. We can apply the Riemann mapping theorem to conjugate the immediate basin of attraction to the unit disk and send the attracting fixed point $p$ to $0$.
Hence we should study the proper holomorphic self maps of the unit disk.
Moreover, the proper maps on the unit disk can be written as 
$
g(z)=e^{i\theta}\Pi_{j=1}^m \frac{z-a_j}{1-\overline{a_j}z}
$ 
for $e^{i\theta}$ on the unit circle $\partial\triangle$, and $m\geq2$ since there is at least one critical point inside $\Omega$ (see Theorem 2.2 on page 59 in \cite{RefCG}), and constants $\lvert a_j\rvert<1$ with at least one of the $a_j=0$ since $f$ has an attracting fixed point $p$ which is sent to $0$.
Note that the degree of $g$ depends on how many critical points are inside $\Omega$, so $2\leq m\leq N.$ Therefore, instead of considering the orbits of polynomials on the attracting basin with an attracting fixed point at $p$, we only need to consider the orbits of Blaschke products on the unit disk with an attracting fixed point at $0$. 

First, we discuss the simplest case when all $a_j=0.$ Then for $g=e^{i\theta} z^m$, we have the following theorem.

\begin{theorem}\label{the4}
	Suppose $g(z)=e^{i\theta}z^m, m\geq2$, we pick a point $\hat{p}\in\triangle\setminus\{0\}$. Then there exists a constant $C_0>0$ such that for every point $z_0\in\triangle$, there exists $q\in \cup_k g^{-k}(\hat{p}), k\geq0$ satisfying $d_\triangle(z_0,q)\leq C_0$, where $d_\triangle$ is the Kobayashi distance on the unit disk $\triangle$.
\end{theorem}

\begin{proof}
	It is easy to describe the dynamics of $g(z)=e^{i\theta}z^m, m\geq2$ since all its iterates
	can be written very explicitly as $g^n(z)=e^{i\frac{1-m^n}{1-m}\theta}z^{m^n}$.
	Let the point $\hat{p}=r e^{i\theta_0}, 0<r<1,$
	then the inverse images of $\hat{p}$ are
	$$T_k :=se^{i\phi}, s=r^{1/m^k}, \phi=\frac{2j\pi +\theta_0-\frac{1-m^k}{1-m}\theta}{m^k},j=0,\dots, m^k-1.$$
	We need to show that all points are at most a finite distance from this sequence in the Kobayashi metric of the disk.

	If $\lvert z_0\rvert \leq \lvert \hat{p}\rvert,$ we can choose $q=\hat{p}$, and by the formula (on page 21 in \cite{RefM}) for the Kobayashi metric on the disk, we have that 
	\begin{equation}\label{eq8}
		d_\triangle(z_0, q)=d_\triangle(z_0,\hat{p})=d_\triangle(0, \frac{z_0-\hat{p}}{1-z_0\overline{\hat{p}}})=\ln \frac{1+\lvert\frac{z_0-\hat{p}}{1-z_0\overline{\hat{p}}}\rvert}{1-\lvert\frac{z_0-\hat{p}}{1-z_0\overline{\hat{p}}}\rvert}.
	\end{equation}
	Let $\sigma=\sup_{\lvert z\rvert\leq\lvert \hat{p}\rvert}\lvert\frac{z-\hat{p}}{1-z\overline{\hat{p}}}\rvert,$ then $0<\sigma<1,$ we obtain
	$$
	d_\triangle(z_0, q)\leq\ln \frac{1+\sigma}{1-\sigma}.
	$$
	Obviously, there is a constant $C'_0:=\ln \frac{1+\sigma}{1-\sigma}$ and a point $q=\hat{p}$ so that $d_\triangle(z_0, q)\leq C'_0$ for all $\lvert z_0\rvert\leq r=\lvert\hat{p}\rvert.$ 
	
	If $\lvert z_0\rvert> \lvert\hat{p}\rvert,$ we write $z_0=\rho e^{i\psi}$ for $\rho> \lvert\hat{p}\rvert, 0\leq\psi<2\pi$. We let $k\geq0$ be given by 
	$$s=r^{1/m^k}\leq \rho<r^{1/m^{k+1}},$$ 
	then let $0\leq j<m^k$ so that 
	$$
	\phi=\frac{2j\pi +\theta_0-\frac{1-m^k}{1-m}\theta }{m^k}\leq \psi\leq \frac{2 \pi (j+1)+\theta_0-\frac{1-m^k}{1-m}\theta}{m^k}.
	$$ 
	We let $q=se^{i \phi},$ then $g^k(q)=\hat{p},$ so $q\in g^{-k}(\hat{p}).$

	If $a<b\leq0,$ then the mean value theorem says there exists $a\leq c\leq b$ so that
	$$
	e^a(b-a)\leq e^b-e^a=(e^x)'_{x=c}(b-a)=e^c(b-a)\leq b-a.
	$$
	Hence 
	$$\Delta r:=\rho-s\leq r^{\frac{1}{m^{k+1}}}-r^{\frac{1}{m^k}}=e^{\frac{\ln r}{m^{k+1}}}-e^{\frac{\ln r}{m^k}}\leq\frac{\ln r}{m^{k+1}}-\frac{\ln r}{m^k}=-\frac{\ln r}{m^{k+1}}.$$

	Moreover, $$\Delta\theta:=\psi-\phi\leq \frac{2 \pi (j+1)+\theta_0-\frac{1-m^k}{1-m}\theta}{m^k} -\frac{2j\pi +\theta_0-\frac{1-m^k}{1-m}\theta }{m^k}=\frac{2\pi}{m^k}.$$ 
	Then we choose a curve $\gamma$ joining $z_0$ to $q$, firstly, following the radius from $z_0=\rho e^{i\psi}$ to $se^{i\psi}$, then walking along the arc on the circle with radius $s$ from $s e^{i \psi}$ to $se^{i \phi}$, see figure \ref{Figure3}. 
	\begin{figure}[htbp]
		\centering 
		\includegraphics[width=0.5\linewidth]{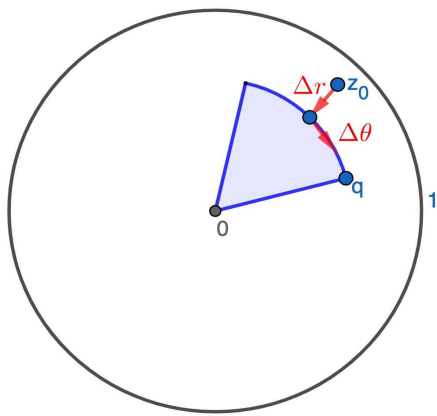}
		\caption{The curve $\gamma$ follows the red arrows from $z_0$ to $q$.}
		\label{Figure3}
	\end{figure}
	Therefore,
	\begin{equation*}
		\begin{aligned}
			d_{\triangle}(z_0, q)\leq&\int_\gamma\frac{\lvert dz\rvert}{1-\lvert z\rvert^2}
			\leq\int_\gamma\frac{\lvert dz\rvert}{1-\lvert z\rvert}
			\leq\int_\gamma\frac{1}{1-r^{\frac{1}{m^{k+1}}}}\lvert dz\rvert\\
			=&\int_\gamma\frac{1}{1-e^{\frac{\ln r}{m^{k+1}}}}\lvert dz\rvert
			=\int_\gamma\frac{1}{e^0-e^{\frac{\ln r}{m^{k+1}}}}\lvert dz\rvert
			=\int_\gamma\frac{1}{(e^x)'_{x=c}\big(0-\frac{\ln r}{m^{k+1}}\big)}\lvert dz\rvert\\
			\leq&\int_\gamma\frac{1}{e^{\frac{\ln r}{m^{k+1}}}\cdot\big(-\frac{\ln r}{m^{k+1}}\big)}\lvert dz\rvert
			\leq\int_\gamma\frac{1}{e^{\frac{\ln r}{m}}\cdot\big(-\frac{\ln r}{m^{k+1}}\big)}\lvert dz\rvert\\
			\leq&-\frac{ m^{k+1}}{\sqrt[m]{r}\ln r} \sqrt{(\Delta r)^2+(\Delta\theta)^2}
			\leq-\frac{ m^{k+1}}{\sqrt[m]{r}\ln r}\frac{\sqrt{(-\ln r)^2+4m^2\pi^2}}{ m^{k+1}}\\
			=&-\frac{1}{\sqrt[m]{r}\ln r }\sqrt{(-\ln r)^2+4m^2\pi^2}:=C''_0.
		\end{aligned}
	\end{equation*}

	In addition, $r$ is a fixed number that we pick for $\hat{p}=r e^{i\theta_0}$. 
	Therefore, there exists a constant $C_0$ and a point $q$ such that $	d_{\triangle}(z_0, q)\leq C_0,$ where we take $C_0=\max\{C'_0, C''_0\}.$

\end{proof}

Next, we generalize Theorem \ref{the4} to general Blaschke products 
$
g(z)=e^{i\theta}\Pi_{j=1}^m \frac{z-a_j}{1-\overline{a_j}z},
$ 
with some $a_j\neq0$. However, there is at least one of the $a_j=0$ since $0$ is a fixed point of $g.$ In this case, the role of $\hat{p}\in\triangle\setminus\{0\}$ in Theorem \ref{the4} can now be $0.$

\begin{theorem}\label{theo2}
	Suppose $$g=e^{i \theta}z^{m_1}\prod_{l=m_1+1}^{m}\frac{z-a_l}{1-\bar{a}_lz}, ~m\geq2,~ 1\leq m_1<m, ~ a_l\neq0, ~a_l\in\triangle,~ e^{i\theta}\in \partial\triangle.$$
	There exists a constant $C_0$ such that for every point $z_0\in\triangle,$ there exists $q\in \cup_k g^{-k}(0)$ satisfying $d_\triangle(z_0,q)\leq C_0$.
\end{theorem}

\begin{proof}
	
	First of all,  if $\lvert z_0\rvert\leq\rho<1$ for some constant $\rho$ to be chosen late, then we choose $q=0.$ 
	Then we have 
	$$
	d_\triangle(z_0, q)=\ln\frac{1+\lvert z_0\rvert}{1-\lvert z_0\rvert}\leq\ln\frac{1+\rho}{1-\rho}:=C'_0.
	$$
	Therefore, there is a constant $C'_0$ such that for every point $z_0\in\triangle$ and $\lvert z_0\rvert\leq\rho$, there exists a $q=0$ satisfying $d_{\triangle}(z_0, q)\leq C'_0$.

	Secondly, if $\lvert z_0\rvert>\rho,$
	then we will show when there still exists a point $q$ such that $d_{\triangle}(z_0, q)$ is uniformly bounded even when $q$ is very close to $\partial\triangle$.
	
	Let us first prove that the preimages of $0$ are dense in the Euclidean metric distance on the boundary of the unit disk, i.e., the set 
	$
	\cup_{n\rightarrow\infty}\{g^{-n}(0)\}
	$ 
	clusters at every point in $\partial\triangle$. 
	First of all, we know that the Blaschke product $g$ carries $\triangle$ onto $\triangle$ and $\hat{\mathbb C}\setminus\bar{\triangle}$ onto $\hat{\mathbb C}\setminus\bar{\triangle}$, and the Julia set of $g$ is the unit circle. 
	See Theorem 1.8 on page 57 and the corresponding  example on page 58 in \cite{RefM}. 
	Secondly, suppose $\tilde{z}_0$ is an arbitrary point on the Julia set on $\partial\triangle$, and we take a small neighborhood $\tilde{\Omega}$ of $\tilde{z}_0$, then 
	$
	\cup_{n\rightarrow\infty}\{g^{n}(\tilde{\Omega})\}
	$ 
	can not avoid more than two points in $\hat{\mathbb C}$ since it is not a normal family. Hence 
	$
	0, z_1:=g^{-1}(0)\neq0
	$ and
	$z_2:=g^{-2}(0)\neq0$ 
	cannot be all outside of 
	$
	\cup_{n\rightarrow\infty}\{g^{n}(\tilde{\Omega})\}.
	$ 
	Suppose
	$
	0\notin\cup_{n\rightarrow\infty}\{g^{n}(\tilde{\Omega})\},
	$ 
	then $z_1$ or $z_2$
	must be inside 
	$
	\cup_{n\rightarrow\infty}\{g^{n}(\tilde{\Omega})\},
	$
	then 
	$
	0\in\cup_{n\rightarrow\infty}\{g^{n+1}(\tilde{\Omega})\}
	$ or $0\in\cup_{n\rightarrow\infty}\{g^{n+2}(\tilde{\Omega})\}$. 
	It implies that there exists a point $\tilde{z}\in\tilde{\Omega}$ and some integer $N$ such that $g^N(\tilde{z})=0$ and $\tilde{z}\in\tilde{\Omega}\cap\triangle.$
	Thus for any $\tilde{z}_0\in\partial\triangle$, there exists a sequence $\{g^{-n}(0)\}$ converges to $\tilde{z}_0.$ Therefore, the preimages of $0$ are dense in $\partial\triangle$.

	Note that this does not finish the proof because the Kobayashi distance from $\{g^{-n}(0)\}$ to points near the boundary might still be arbitrarily large. We still need to show that this is not the case.
	
	Before continuing with the proof of Theorem \ref{theo2}, we have the following two lemmas for the orbits near the boundary of the unit disk.
	\begin{lemma}\label{lem1}(\cite{RefG})
		$\lvert g'(z)\rvert>1$ for all $z\in\partial\triangle$.
	\end{lemma}
	\begin{proof}
		First of all, let $h(z)=e^{i \theta}\prod_{k=1}^{m}\frac{z-a_k}{1-\bar{a}_kz}:= e^{i \theta}h_1h_2\cdots h_m$
		where $h_i=\frac{z-a_k}{1-\bar{a}_kz}, 1\leq i\leq m.$
		Then 
		\begin{equation*}
			\begin{aligned}
				\frac{h'}{h}&=\frac{h'_1}{h_1}+\frac{h'_2}{h_2}+\cdots+\frac{h'_m}{h_m}
				=\sum_{k=1}^{m}\frac{1-a_k\bar{a}_k}{(1-\bar{a}_kz)^2}\cdot\frac{1-\bar{a}_kz}{z-a_k}\\
				&=\sum_{k=1}^m\frac{1-a_k\bar{a}_k}{(1-\bar{a}_kz)(z-a_k)}.\\
			\end{aligned}
		\end{equation*}
		In addition, for all $\zeta\in\partial\triangle,$ we have $\zeta\bar{\zeta}=1$ and $\lvert h(\zeta)\rvert=1.$ Hence
		\begin{equation}{\label{eq9}}
			\begin{aligned}
				\zeta\frac{h'(\zeta)}{h(\zeta)}
				=\sum_{k=1}^m\frac{1-\lvert a_k\rvert^2}{\bar{\zeta}{(1-\bar{a}_k\zeta)(\zeta-a_k)}}
				=\sum_{k=1}^m\frac{1-\lvert a_k\rvert^2}{\lvert \zeta-a_k\rvert^2}.
			\end{aligned}
		\end{equation} 
		Since $g(z)=e^{i \theta}z^{m_1}\prod_{l=m_1+1}^{m}\frac{z-a_l}{1-\bar{a}_lz}$ and the equation (\ref{eq9}), we know some $a_k=0$, then $$\lvert g'(z)\rvert=m_1+\sum_{l=m_1+1}^{m}\frac{1-\lvert a_l\rvert^2}{\lvert z-a_l\rvert^2}>m_1\geq1,$$
		for all $z\in\partial\triangle.$
		
	\end{proof}

	Therefore, we can choose $r_0\in(0, 1)$ so that $\lvert g'(z)\rvert>1+\epsilon$ for any $z$ inside the annulus $\mathbb A:=\{z: r_0<\lvert z\rvert<1\}$ and some fixed $\epsilon>0$.
	
	\begin{lemma}\label{lem5}
		$\lvert g^{-1}(z)\rvert>\lvert z\rvert$ for all $z\in\mathbb A.$ 
	\end{lemma}
	
	\begin{proof}

		We choose a point $\omega=e^{i\theta}\in\partial\triangle$ and draw a straight line $L$ from $\partial\mathbb A$ to $\omega$ such that $L$ is perpendicular to $\partial\mathbb A$. Let $\gamma(t):=te^{i\theta}, r_0< t\leq1$ be a subline of $L$ and $\eta$ be any preimage of $\omega$, i.e., $\eta=g^{-1}(\omega)$. We consider the branch of $g^{-1}(\gamma(t))$ satisfying with $g^{-1}(\omega)=\eta.$

		Let $t_0:=1-\delta$ be the minimum constant such that $g^{-1}(\gamma(t))\in\mathbb A$ for all $t\in (t_0, 1).$
		We know that $0<\lvert \big(g^{-1}(z)\big)'\rvert<1/(1+\epsilon)<1$ for $z\in\mathbb A$, and hence the Euclidean distance from $g^{-1}(z)=g^{-1}(\gamma(t))$ to the boundary is 
		\begin{equation*}
			\begin{aligned}
				1-\lvert g^{-1}(z)\rvert
				&\leq d_E(g^{-1}(z), g^{-1}(e^{i\theta}))\leq \int_{\lvert z\rvert}^{1}\lvert(g^{-1}(\gamma(t)))'\rvert\lvert\gamma'(t)\rvert dt\\
				&\leq\int_{\lvert z\rvert}^{1}\frac{1}{1+\epsilon}\lvert\gamma'(t)\rvert dt
				=\frac{1}{1+\epsilon}\int_{\lvert z\rvert}^{1}\lvert e^{i\theta}\rvert dt=\frac{1-\lvert z\rvert}{1+\epsilon}
				<1-\lvert z\rvert,
			\end{aligned}
		\end{equation*}
		hence $\lvert g^{-1}(z)\rvert>\lvert z\rvert$ for all $g^{-1}(\gamma(t))\in\mathbb A, t\in(t_0, 1)$.

		If $t_0=1-\delta=r_0,$ then the proof is done. If $t_0=1-\delta>r_0,$ then suppose $t$ is a little bit smaller than $t_0$, i.e., $t=t_0-\epsilon$ for sufficient small $\epsilon$, then  $\lvert g^{-1}(\gamma(t))\rvert=\lvert g^{-1}(\gamma(t_0-\epsilon))\rvert=\lvert g^{-1}(\gamma(t_0)\rvert-\epsilon'\geq t_0-\epsilon'$ for a small $\epsilon'$ by continuity of $g^{-1}(t)$. Hence $g^{-1}(\gamma(t_0-\epsilon))\in\mathbb A.$ 
		It implies that $g^{-1}(\gamma(t))$ still can be continued to $I:=(t_0-\epsilon, t_0]$. 
		Then it contradicts that $t_0$ is the minimum constant such that $g^{-1}(\gamma(t))\in\mathbb A.$ 
		Therefore, $t_0=r_0$ which means 
		that $\lvert g^{-1}(z)\rvert>\lvert z\rvert$ for all $z\in\mathbb A$. This finishes the proof of Lemma \ref{lem5}.

	\end{proof}

	We continue with the proof of Theorem \ref{theo2}. We consider the inverse orbit of $0$ near the boundary of the unit disk precisely, especially we investigate the inverse orbit inside $\mathbb A.$ 
	
	Let $D_0=g^{-1}(\triangle\setminus\mathbb A)$, and choose $r_1>r_0$ such that $D_0\subset\mathbb B_0=\{z: \lvert z\rvert\leq r_1<1\}.$ Let $D_1=g^{-1}(\mathbb B_0),$ then we can choose $r_2>r_1$ such that $D_1\subset\mathbb B$ where $\mathbb B=\{ z: \lvert z\rvert<r_2<1\}.$ Let $\mathbb B_1:=\{z: r_1<\lvert z\rvert<r_2\}$, then for any $z_0\in\triangle\setminus\mathbb B,$ there exists a point $z'_0:=g^{M_1}(z_0)\in\mathbb B_1,$ $z'_1:=g^{M_1-1}(z_0)\in \triangle\setminus\mathbb B$ for some integer $M_1$ (See figure \ref{Figure2}).
	\begin{figure}[htbp]
		\centering 
		\includegraphics[width=0.6\linewidth]{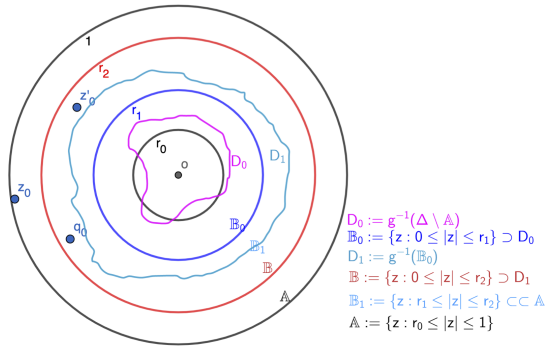}
		\caption{Domains near the boundary of the unit disk.}
		\label{Figure2}
	\end{figure}
	Since the preimages of $0$ are dense in the boundary of the unit disk, we can calculate all preimages of $0$ and choose some integer $M_2$ such that $q_0:=g^{-M_2}(0)\in\mathbb B_1, q_1:=g^{-1}(q_0)\in\triangle\setminus\mathbb B.$
	
	Now, we need to choose a sector to be $\mathbb A^+$ such that $z'_0, q_0\in \mathbb A^+$, but $z'_0, q_0$ are neither on $\partial\mathbb A^+$ or very close to $\partial\mathbb A^+$:
	first, choose a sector $S$ with angle $\frac{\pi}{3}$ and one boundary line of $S$ passes through $0$ and $z'_0$ such that $q_0\notin S.$ Secondly, choosing a sector $\tilde{S}$ inside $S$ with angle $\frac{\pi}{5}$ such that $\tilde{S}$ and $S$ have the same angle bisector region of $\mathbb A$ not including $z'_0$ and $q_0.$ Then we let $\mathbb A^+:=\mathbb A\backslash(\tilde{S}\cap\mathbb A).$

	Since we know precisely the orbit $\{z_n\}$, we can choose the inverse map to be the branch of $g^{-1}(z)$ on $\mathbb A^+$ such that  $g^{-1}(z'_0)=g^{-1}(g^{M_1}(z_0))=g^{M_1-1}(z_0)=z'_1$ sending $q_0$ to $q_1$. In addition, there are no critical points inside $\mathbb A$ since $\lvert g'(z)\rvert>1+\epsilon$ and $0<\lvert\big(g^{-1}(z)\big)'\rvert<\frac{1}{1+\varepsilon}$ for all $z\in\mathbb A.$
	Then the branch of the inverse map $g^{-1}(z)$ which we choose is well defined in $\mathbb A^+.$ 
	Inductively, we can, in the same way, define $g^{-n}(z)$ in $\mathbb A^+.$
	We obtain the inverse map $g^{-n}(z): \mathbb A^+\rightarrow\mathbb A$ is analytic. And the annulus $\triangle\setminus\mathbb B$ contains $z'_1, z'_2, z'_3, \cdots, z'_n$ and $q_1, q_2, q_3, \cdots, q_n.$
	
	Next, we estimate the Kobayashi distance. 
	Choosing a sector $\bar{S}\subset\tilde{S}$ with an angle $\frac{\pi}{4},$ then
	let $\mathbb B_1^+:=\mathbb B_1\backslash(\bar{S}\cap\mathbb B_1)\subset\mathbb A^+$, then $\mathbb B_1^+\subset\subset\mathbb A$ including $z'_0, q_0.$ Hence there is a constant $C_0$ such that 
	\begin{equation}\label{eq3.1}
		d_{\mathbb B_1^+}(z'_0, q_0)\leq C_0.
	\end{equation}
	In addition, by Corollary \ref{cor3}, we have
	\begin{equation}\label{eq3.2}
		d_{\mathbb A^+}(z'_0, q_0)< d_{\mathbb B_1^+}(z'_0, q_0).
	\end{equation}
	Inductively, $g^{-M_1}(z): \mathbb A^+\rightarrow\mathbb A$ sends the point $z'_0, q_0\in \mathbb A^+$ to the point $z_0, q_{M_1}\in\mathbb A,$ respectively. The Kobayashi distance is decreasing by Proposition \ref{pro2}, then we obtain 
	\begin{equation}\label{eq3.3}
		d_{\mathbb{A}}(z_0, q_{M_1})\leq d_{\mathbb A^+}(z'_0, q_0).
	\end{equation}
	Since $\mathbb A\subset \triangle$, using Corollary \ref{cor3} once again, we know 
	\begin{equation}\label{eq3.4}
		d_\triangle(z_0, q_{M_1})\leq d_{\mathbb{A}}(z_0, q_{M_1}). 
	\end{equation}
	Therefore, from equations (\ref{eq3.1}) to (\ref{eq3.4}), there is a constant $C_0$ and a point $q:=q_{M_1}\in\cup_{k}g^{-k}(0)$ such that
	\begin{equation}
		d_\triangle(z_0, q)\leq d_{\mathbb{A}}(z_0, q)=d_{\mathbb A}(z_0, q_{M_1})\leq d_{\mathbb A^+}(z'_0, q_0)<d_{\mathbb B_1^+}(z'_0, q_0) \leq C_0,
	\end{equation}
	for some integer $k.$ 
	Choosing $\rho=r_2$, we finally prove the theorem. 
	
\end{proof}

Now, we can obtain our main theorems as follows:
\begin{theorem}\label{theom1}
	Suppose $f(z)$ is a polynomial of degree $N\geq2$ on $\mathbb{C}$, $\Omega$ is an immediate attracting basin of $f(z)$ and $p$ is an attracting fixed point inside $\Omega$, $\{f^{-1}(p)\}\cap \Omega\neq\{p\}$. 
	Then there is a constant $C$ such that for every point $z_0\in \Omega$, there exists a point $q\in \cup_k f^{-k}(p), k\geq0$ so that $d_\Omega(z_0, q)\leq C,  d_\Omega$ is the Kobayashi distance on $\Omega$. 
\end{theorem}

This theorem would imply that the orbit of $z_0$ behaves like the orbit of some preimage of the fixed point $p$. More precisely, the orbit of $z_0$ is shadowed by the orbit of the point $f^{-k}(p)$
for some positive integer $k$. Hence it would be enough to consider only these preimage points $f^{-k}(p)$. But then, instead of considering the forward orbit of the preimage of the fixed point $p$, we can equivalently study the backward orbit of the fixed point $p$. Note that this is very useful for practical purposes because these inverse orbits can be color plotted. 

\begin{proof}
	Since $\Omega$ is an immediate attracting basin of $f(z)$, $p$ is an attracting fixed point inside $\Omega$, we know that $f(\Omega)=\Omega$ and $\Omega$ is simply connected, then by the Riemann mapping theorem,
	we can conjugate $f$ to $g=e^{i \theta}\prod_{j=1}^{m}\frac{z-a_j}{1-\bar{a}_jz},$ which is a proper self map of the unit disk
	$\triangle$, with an attracting fixed point at the origin, and $g$ is a finite Blaschke product. We refer the reader to Lemma 15.5 on page 163 in \cite{RefM} for more details.
	Then this theorem is true because of Theorem \ref{theo2}.

\end{proof}

In the remainder of this section, we will prove that Theorem \ref{theom1} still holds when $\Omega$ is the whole basin of attraction.

\begin{theorem}\label{the3}
	Suppose $f(z)$ is a polynomial of degree $N\geq 2$ on $\mathbb{C}$, $p$ is an attracting fixed point of $f(z),$ $\Omega_1$ is an immediate basin of attraction of $p$, $\{f^{-1}(p)\}\cap \Omega_1\neq\{p\}$,
	$\mathcal{A}(p)$ is the basin of attraction of $p$, $\Omega_i (i=1, 2, \cdots)$ are the connected components of $\mathcal{A}(p)$. Then there is a constant $\tilde{C}$ so that for every point $z_0$ inside any $\Omega_i$, there exists a point $q\in \cup_k f^{-k}(p)$ inside $\Omega_i$ such that $d_{\Omega_i}(z_0, q)\leq \tilde{C}$, where $d_{\Omega_i}$ is the Kobayashi distance on $\Omega_i.$  
\end{theorem} 

\begin{proof}
	
	Since $f$ is a polynomial with an attracting fixed point $p\in\Omega_1$. Then the connected components of $\mathcal{A}(p)$ have only two situations:
	
	(1) $\mathcal{A}=\Omega_1.$ Then this theorem is essentially as same as Theorem \ref{theom1}.

	(2) $\mathcal{A}$ has at least two connected components, then $\mathcal{A}$ has infinitely many connected components.
	Suppose $\Omega_2$ is another connected component of $f$ which is distinct  from $\Omega_1$ and $f(\Omega_2)=\Omega_1$. We know $f(\Omega_1)=\Omega_1$, and then there must have the third component $\Omega_3$ which can be mapped to $\Omega_2$, and so on. It implies that $f$ has infinitely many connected components.

	For the second situation, let us first consider the orbit between two connected components, $\Omega_1$ and $\Omega_2,$ where $f(\Omega_2)=\Omega_1.$ If the start point $z_0$ is inside $\Omega_1,$ then the proof is done by Theorem \ref{theom1}. If the start point $z_0$ is inside $\Omega_2,$ then there is a point $\hat{z}_0:=f(z_0)\in\Omega_1.$ 
	By the above Theorem \ref{theom1}, we know that there is a constant $C$ such that for every point $\hat{z}_0\in \Omega_1$, there exists a point $\hat{q}\in \cup_k f^{-k}(p), k\geq0$ in $\Omega_1$ so that $d_{\Omega_1}(\hat{z}_0, \hat{q})\leq C,  d_{\Omega_1}$ is the Kobayashi distance on $\Omega_1$. 
	Next, we need to show that there is a constant $C'$ such that for every point $z_0\in \Omega_2$, there exists a point $q:=f^{-1}(\hat{q})\in\cup_k f^{-k}(p), k\geq0$ in $\Omega_2$ so that $d_{\Omega_2}(z_0, q)\leq C'.$

	Since $\Omega_1$ and $\Omega_2$ are simply connected, by the Riemann mapping theorem, there are two biholomorphic maps, $\psi_1: \Omega_1\rightarrow \triangle$ and $\psi_2: \Omega_2\rightarrow\triangle$. Then $f$ is conjugate with $g=\psi_1\circ f\circ \psi_2^{-1}$, which is a proper self-map of the unit disk $\triangle$. Hence $g$ is a Blaschke product. Then these three points $z_0\in\Omega_2, \hat{z}_0=f(z_0)\in\Omega_1, \hat{q}\in\Omega_1$ are sent to $\triangle$, we denote them by $Z_0, \hat{Z}_0, \hat{Q}\in\triangle,$ respectively, and $d_\triangle(\hat{Z}_0, \hat{Q})= d_{\Omega_1}(\hat{z}_0, \hat{q})<C$.  
	Therefore, it is equivalent to prove that there exists a point $Q=g^{-1}(\hat{Q})\in\triangle$ and a constant $C'$ (independent of $z_0$) such that $d_\triangle(Z_0, Q)<C'$. 
	
	We know that $g$ has finitely many critical points. We can choose a disk $\triangle(0, r_0)=\{z: \lvert z\rvert\leq r_0<1\}$ including all critical points of $g$. We denote $D_0=g(\triangle(0, r_0))$ and choose a disk $\triangle(0, R_0)=\{z: \lvert z\rvert\leq R_0<1\}$ such that $D_0\subseteq\triangle(0, R_0)$, then let $D_1:=\triangle\setminus\triangle(0, R_0)$, see the following Figure \ref{Figure5}.
	\begin{figure}[htbp]
		\centering 
		\includegraphics[width=0.6\linewidth]{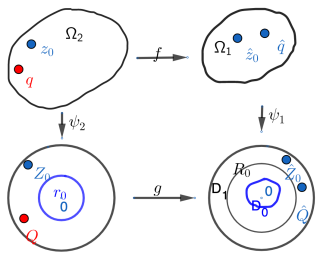}
		\caption{The conjugation.}
		\label{Figure5}
	\end{figure}
	Then there are four cases for distributing $\hat{Z}_0$ and $\hat{Q}.$
	
	Case 1: if $\hat{Z}_0, \hat{Q}\in \triangle(0, R_0),$ choose a disk  $\triangle(0, r_1)=\{z: \lvert z\rvert\leq r_1<1\}$ such that $g^{-1}(\triangle(0, R_0))\subseteq \triangle(0, r_1)$, then $Z_0\in\triangle(0, r_1)$ and there exists a point $Q=g^{-1}(\hat{Q})\in\triangle(0, r_1)$. Hence  $d_\triangle(Z_0, Q)\leq C'$ for some uniform constant $C'.$

	Case 2: if $\hat{Z}_0\in \triangle(0, R_0), \hat{Q}\in D_1,$ then there exists a disk $\triangle(0, R_1)=\{z: \lvert z\rvert \leq R_1<1\}$ including $\hat{Q}$ and $\hat{Z}_0$ since $d_\triangle(\hat{Z}_0, \hat{Q})\leq C.$ Hence there is  a disk $\triangle(0, r_2)=\{z: \lvert z\rvert\leq r_2<1\}$ so that $g^{-1}(\triangle(0, R_1))\subseteq\triangle(0, r_2)$. Then letting $Q$ be any point of $g^{-1}(\hat{Q})$ inside $\triangle(0, r_2).$ Hence $d_\triangle(Z_0, Q)\leq C''$ for some uniform constant $C''$.

	Case 3: if $\hat{Z}_0\in D_1, \hat{Q}\in \triangle(0, R_0),$ this situation is the same as case 2. 
	
	Case 4: if $\hat{Z}_0, \hat{Q}\in D_1,$ then $g^{-1}(z)$ is locally holomorphic from a subset of $D_1$ to $\triangle.$ We can choose the branch of the inverse map of $g$ such that $Z_0=g^{-1}(\hat{Z}_0)\in\triangle\setminus\triangle(0, r_0)$, then there exists a point $Q:=g^{-1}(\hat{Q})\in\triangle\setminus\triangle(0, r_0).$ Next, we need to show that the Kobayashi distance $d_{\triangle}(Z_0, Q)$ is still uniformly bounded. The way to prove this is similar to the proof of Theorem \ref{theo2}.
	
	If $r_0\leq\lvert Z_0\rvert\leq\rho<1, r_0\leq\lvert Q\rvert \leq\rho<1$, then calculating similarly to equation (\ref{eq8}), we conclude that there is a constant $C'$ such that $d_{\triangle}(Z_0, Q)\leq C'$. But if either $\lvert Z_0\rvert>\rho$ or $\lvert Q\rvert>\rho$, 
	we can prove that $d_{\triangle}(Z_0, Q)$ is still uniformly bounded as follows.
	
	We choose a sector to be $D_1^+$ inside $D_1$ such that $\hat{Z}_0, \hat{Q}\in D_1^+$, but $\hat{Z}_0, \hat{Q}$ are neither on $\partial D_1^+$ or very close to $\partial D_1^+$. The way to choose $D_1^+$ is the same as choosing $\mathbb A^+$ in the proof of Theorem \ref{theo2}.  
	Then we have $$d_{\triangle}(Z, Q)\leq d_{D_1^+}(\hat{Z}_0, \hat{Q})$$ since $g^{-1}(z): D_1^+\rightarrow\triangle$ is holomorphic.

	Now, we need to show that $d_{D_1^+}(\hat{Z}_0, \hat{Q})$ is bounded by some constant. In Wold's paper \cite{RefW} (Theorem 3.4), he proved that 
	$F_{D_1^+}(z, \xi_1)-F_{\triangle}(z, \xi_2)=O(\delta(z))$, where $\delta$ denotes the boundary distance.  
	Then 
	\begin{equation*}
		\begin{aligned}
			d_{D_1^+}(Z_0, Q)&\leq d_{\triangle}(Z_0, Q)+d_E(\gamma(t), \partial\triangle)\leq C+\int_{0}^{1}O(\lvert\gamma(t)-\frac{\gamma(t)}{\lvert\gamma(t)\rvert}\rvert)\gamma'(t)dt\\
			&\leq C+\int_{0}^{1}\lvert\gamma'(t)\rvert dt\leq C+\lvert Z_0-Q\rvert< C+2.
		\end{aligned}
	\end{equation*}
	
	Hence there exists a constant $C':=C+2$ such that $d_{\triangle}(\hat{Z}, \hat{Q})\leq d_{D_2}(Z_0, Q)<C'.$

	Therefore, this theorem is true for all these four cases, i.e., there is a constant $\tilde{C}$ so that there exists a point $q=f^{-1}(\hat{q})\in\Omega_2$ such that  $d_{\Omega_2}(z_0, q)\leq \tilde{C}.$

	Let us continuously consider the orbit between more connected components for the situation (2).
	Suppose the starting point $z_0$ is inside some connected component $\Omega_{i_1}, i_1= 3, 4, 5,\cdots$.
	Then there is a positive integer $N_0$ such that $\hat{z}_0:=f^{N_0}(z_0)\in\Omega_{1}$. Note that if $N_0=1,$ it is the same as the orbit between two connected components, so the rest is to consider when $N_0\geq 2.$
	By Theorem \ref{theom1}, we  have that there is a constant $C$ such that for every point $\hat{z}_0\in \Omega_{1}$, there exists a point $\hat{q}\in \cup_k f^{-k}(p), k\geq0$ in $\Omega_1$ so that $d_{\Omega_{1}}(\hat{z}_0, \hat{q})\leq C.$ Then we only need to show that there is a point $q\in\Omega_{i_1}$ such that $d_{\Omega_{i_1}}(z_0, q)$ is uniformly bounded. When finding $q$, which is some point of iterating the inverse of $\hat{q}$, we need to be careful in dealing with the critical points when it appears in the inverse orbit.
	
	We know that $f$ has finitely many critical points, so there are only finitely many $\Omega_i$ containing critical points. Let $\Omega_{i_2}$ be a connected component satisfies $f(\Omega_{i_2})=\Omega_{1}$ and $f^{N_0-1}(\Omega_{i_1})=\Omega_{i_2}$. If there are some critical points in $\Omega_{{i_2}}$, then we do the same procedure as above to find that there is a point $f^{-1}(\hat{q})$ such that $d_{\Omega_{i_2}}(f^{-1}(\hat{z}_0), f^{-1}(\hat{q}))$ is uniformly bounded. If there are no critical points inside $\Omega_{{i_2}}$, then the Kobayashi metric is an isometry. Inductively, after $N_0$ times of iterating the inverse orbit of $\hat{q}$, we definitely can find a point $q:=f^{-N_0}(\hat{q})\in\Omega_{i_1}$ such that $d_{\Omega_{i_1}}(z_0, q)$ is uniformly bounded. 
	
	Therefore, no matter how many connected components of $\mathcal{A}(p)$ has, there is a constant $\tilde{C}$ so that for every point $z_0$ inside any $\Omega_i,$ there exists a point $q\in \cup_k f^{-k}(p)$ inside $\Omega_i$ such that $d_{\Omega_i}(z_0, q)\leq \tilde{C}$, where $d_{\Omega_i}$ is the Kobayashi distance on $\Omega_i.$  
	
\end{proof}

\bmhead{Acknowledgments}
The author is very grateful to her advisor, Professor John Erik Forn\ae ss, for suggesting this research problem and his patient guidance and valuable comments. In addition, the author thanks the University of Parma for supporting her doctoral study. She also appreciates the staff of NTNU in Norway for enabling her to visit the Department of Mathematical Sciences so that this research works well.

\end{document}